\documentclass[11pt,reqno,b5paper]{amsproc}
\usepackage{color,anysize,float}
\usepackage{citeref}
\usepackage[noadjust]{cite}

\usepackage{enumerate}

\usepackage[colorlinks,urlcolor=blue,citecolor=blue,linkcolor=blue]{hyperref}

\usepackage{graphicx}%

\marginsize{15mm}{15mm}{1cm}{1cm}
\renewcommand{\bibitempages}[1]{}

\newtheorem{thm}{\sc Theorem}[section]

\newtheorem{lem}[thm]{\sc Lemma}
\newtheorem{prop}[thm]{\sc Proposition}
\newtheorem{remark}{Remark}[section]
\theoremstyle{definition}

\theoremstyle{remark}

\numberwithin{equation}{section}
\allowdisplaybreaks
\everymath{\displaystyle}

\begin{document}
	
	\title[
	Schrödinger--Bopp--Podolsky System with Indefinite Potential]
	{On the Schrödinger--Bopp--Podolsky system with indefinite potential: ground states, multiplicity and exponential decay}
	\author[T. Xiao, F. Wang, L.-F. Yin]{Ting Xiao\,$^\mathrm{a}$, Fan Wang\,$^\mathrm{a}$, Li-Feng Yin\,$^\mathrm{b*}$\vspace{-1em}}
	\dedicatory{$^\mathrm{a}$Department of Mathematics, Southwest Jiaotong University\\
		Chengdu, 611756, China \\$^\mathrm{b*}$School of Mathematical Sciences, Sichuan Normal University\\
		Chengdu, 610066, China}
	\thanks{$^*$Corresponding author. \\
		$^\dag$Emails: \tt 19960394045@163.com; wangf767@swjtu.edu.cn; yin136823@163.com}
	\maketitle
	\begin{abstract}
		In this paper, we study the Schrödinger--Bopp--Podolsky system
		\begin{equation*}
			\begin{cases}
				-\Delta u + V(x)u + \phi u = f(x,u), & \text{in } \mathbb{R}^3, \\
				-\Delta \phi + a^2 \Delta^2 \phi = 4\pi u^2, & \text{in } \mathbb{R}^3.
			\end{cases}
		\end{equation*}
		We consider the case where the potential \(V\) is indefinite so that the Schrödinger operator \(-\Delta + V\) has a finite-dimensional negative space. Under suitable assumptions on the potential \(V\) and nonlinearity $f(x,u)$, we prove the existence of nontrivial solutions via a local linking argument and Morse theory. Moreover, these solutions are shown to decay exponentially at infinity. Additionally, a ground state solution is obtained by minimization techniques. Finally, if \(f(x,u)\) is odd with respect to \(u\), we obtain an unbounded sequence of solutions using the symmetric mountain pass theorem.\\
		
	\noindent{\bf Keywords:} Schrödinger--Bopp--Podolsky system, indefinite potential, local linking argument, Morse theory\\
	{\bf Mathematics Subject Classification:} 35A15, 35B38, 49J45
\end{abstract}

	\section{Introduction}
	In this paper, we consider the Schrödinger field system coupled with its own electromagnetic field in the purely electrostatic case of the Bopp--Podolsky electromagnetic theory, described by the following system (see \cite{MR3957980})
	\begin{equation}\label{eq:2}
		\begin{cases}
			i\hbar \frac{\partial \psi}{\partial t} = -\frac{\hbar^2}{2m} \Delta \psi + U(x)\psi + q^2\phi \psi - \tilde{f}(x, |\psi|)\psi, \quad (t,x) \in \mathbb{R}^+ \times \mathbb{R}^3, \\
			-\Delta \phi + a^2 \Delta^2 \phi = 4\pi |\psi|^2, & 
		\end{cases}
	\end{equation}
	where $\hbar$ is the Planck constant, $m$ is the meson mass, $q>0$ is the charge, $\psi(t,x):\mathbb{R}^+ \times \mathbb{R}^3 \to \mathbb{C}$ is a complex-valued function and $a>0$ is the Bopp--Podolsky parameter that has the dimension of inverse mass. For simplicity, we set $\hbar^2 = 2m$ and $q=1$ in what follows. Further physical details of this model can be found in \cite{MR3689587,MR3998534,PhysRevD.90.085012,MR3707699}. In particular, system \eqref{eq:2} admits standing wave solutions of the form \[\psi(t,x)=e^{-i\omega t/\hbar}u(x),\] where $\omega \in \mathbb{R}$ denotes the frequency. Then system \eqref{eq:2} can be transformed into the following  nonlinear Schr\"odinger--Bopp--Podolsky (SBP) system
	\begin{equation}\label{eq:4}
		\begin{cases}
			-\Delta u + V(x)u + \phi u = f(x,u), & \text{in } \mathbb{R}^3, \\
			-\Delta \phi + a^2 \Delta^2 \phi = 4\pi u^2, & \text{in } \mathbb{R}^3,
		\end{cases}
	\end{equation}
	where
	\[V(x) = U(x)-\omega   \ \text{,} \  f(x,t)=\tilde{f}(x, |t|)t.\]
	
	System \eqref{eq:4} admits a variational structure. Under appropriate assumptions, it is well established that one can associate with it the energy functional $\Phi: H^1(\mathbb{R}^3)\times\mathcal{D} \to \mathbb{R}$. For the definition and properties of $\mathcal{D}$, see Section 2.
	\begin{equation*}\label{f2.2}
		\Phi(u,\phi)=\frac{1}{2}\int \left(|\nabla u|^2+V(x)u^2\right)+\frac{1}{2}\int\phi u^2-\frac{1}{16\pi}\int|\nabla \phi|^2-\frac{a^2}{16\pi}\int|\Delta \phi|^2-\int F(x,u)
	\end{equation*}
	with $ F(x,t) = \int_0^t f(x,\tau) \, d\tau$, so that $(u,\phi)$ solves \eqref{eq:4} if and only if it is a critical point of $\Phi$.
	However, owing to the presence of the $\nabla\phi$ term, the functional $\Phi$ is strongly unbounded from below and above under any conditions on $V$. Hence, to solve this problem, d'Avenia and Siciliano \cite{MR3957980} adopted a reduction
	procedure. For any fixed $u(x) \in H^1(\mathbb{R}^3)$, the Riesz Theorem guarantees a unique solution of the second equation in system \eqref{eq:4} in $\mathcal{D}$, namely $\phi=\phi_u=\mathcal{K}*|u|^2$, where $*$ represents the convolution in $\mathbb{R}^3$ and
	\[\mathcal{K}(x)=\frac{1-e^{-\frac{|x|}{a}}}{|x|}.\]
	Then we define the  functional $J:H^1(\mathbb{R}^3)\to \mathbb{R}$ by
	\begin{equation}\label{f2.3}
		J(u):=\Phi(u,\phi_{u})=\frac{1}{2}\int \left(|\nabla u|^2+V(x)u^2\right)+\frac{1}{4}\int\phi_u u^2-\int F(x,u),
	\end{equation}
	as shown in \cite{MR3957980}, if $u$ is a critical point of $J$ then $(u, \phi_u)$ is a solution of \eqref{eq:4}.

	In the past few years, the existence of nontrivial solutions for the following equation 
	\begin{equation}\label{eq:sec1}
		\begin{cases}
			-\Delta u + V(x)u + q^2\phi u = g(u), & \text{in } \mathbb{R}^3, \\
			-\Delta \phi + a^2 \Delta^2 \phi = 4\pi u^2, & \text{in } \mathbb{R}^3,
		\end{cases}
	\end{equation}
has been extensively studied via variational methods. In these studies, $V(x)$ was taken either as a constant or non-constant. In the first case, d'Avenia and Siciliano \cite{MR3957980} considered $g(u)=|u|^{p-2}u$ with $V(x)=V_\infty$. 
They employed the splitting lemma and monotonicity trick to prove that system \eqref{eq:sec1} admits a nontrivial solution if $p \in (2,6)$ and $q$ is sufficiently small, and also if $p \in (3,6)$ and $q \neq 0$; when $p \ge 6$, by applying the Pohožaev identity, they showed that system \eqref{eq:sec1} has no nontrivial solutions. Moreover, they also showed that, in the radial case, the solutions they found tend to solutions of the classical Schr\"odinger--Poisson system as $a \to 0$.
Subsequently, Siciliano and Silva \cite{MR4119258} complemented and improved the results in \cite{MR3957980} via the so-called “fibering method”. More precisely, they showed that the system \eqref{eq:sec1} admitted no solution when $p \in (2,3]$ and $q$ was sufficiently large, whereas it possessed two radial solutions when $q$ was sufficiently small. For more results, the reader is referred to \cite{Liu2022,MR4119258,MR4982514,MR5014952}. 

In the second case, $V(x)$ is not a constant. Yang et al. \cite{MR4143620} considered $g(u)=\lambda f(u) +|u|^4u$ without imposing any growth conditions or the Ambrosetti–Rabinowitz condition and assumed $V\in C^1(\mathbb{R}^3,\mathbb{R})$, $V_0 = \inf_{x \in \mathbb{R}^3} V(x) > 0$. Under appropriate assumptions on $g(u)$, they established the existence of a nontrivial solution of system \eqref{eq:sec1} by using a cut-off function, the mountain pass theorem and Moser iteration.
 Teng and Yan \cite{MR4198483} considered the case $g(u)=Q(x)|u|^{p-1}u$ with $p\in (3,5)$, $\lim_{|x|\to \infty}Q(x)=Q_\infty >0$ and assumed $V(x)=V_\infty+a(x)$, where $\lim_{|x|\to\infty}V(x)=V_\infty>0$ and $a(x)\in L^{\frac{3}{2}}(\mathbb{R}^3)$, $a(x)>0$. By developing some energy estimates and using a topological argument involving the barycenter function, they obtained the existence of a positive bound state solution of system \eqref{eq:sec1} for $p\in(3,5)$ under appropriate assumptions on $g(u)$. More recently, Chen et al. \cite{ChenLiRadulescuTang2021} developed a more general framework for the non-autonomous case and proved the existence of a ground state solution of system \eqref{eq:sec1} under mild assumptions on $V(x)$ and $g(u)$. Their approach relied on some original analytic techniques and new estimates of the ground state energy, with $V(x) \in C(\mathbb{R}^3,[0,+\infty))$ and $V_\infty=\lim_{|y|\to\infty}V(y)=\sup_{x\in\mathbb{R}^3}V(x)>0$ (the precise assumptions on $V$ are omitted here for brevity). Furthermore, Zheng et al. \cite{MR4944207} considered that $V_0=\inf_{x\in \mathbb{R}^3}V(x)>0$ and $g(u)=f(x,u)$ satisfied asymptotically cubic conditions rather than super‑cubic conditions. They proved the existence of a radial ground state sign‑changing solution $u$ of system \eqref{eq:sec1} with exactly two nodal domains. Moreover, they further showed that the energy of any radial sign‑changing solution of system \eqref{eq:sec1} was strictly greater than twice the least energy. For more results, the reader is referred to \cite{MR4626835,MR4519835,MR4050783,MR4217993}. 
 
 In those papers, it is required that the potential is positive definite, that is 
 \begin{equation*}
 V_0 = \inf_{x \in \mathbb{R}^3} V(x) > 0.
 \end{equation*}
 In those cases, the zero function is a local minimum point of the functional, so the mountain pass theorem can be used.
However, since $\omega\in\mathbb{R}$ can take arbitrarily large positive values, $V(x)$ may become negative. In this case, $V(x)$ is no longer positive definite but indefinite, and the zero function is not a local minimizer of the energy functional. Consequently, the classical mountain pass theorem is not applicable. Moreover, in this case the functional $J$ would not enjoy the general linking geometry due to the term $\phi_{u}$. Hence, the classical linking theorem is also not applicable. To our knowledge, no such results exist for the SBP system with an indefinite potential. However, when $a = 0$, the SBP system \eqref{eq:4} reduces to the Schr\"odinger--Poisson system, for which several works have investigated the existence of nontrivial solutions in the presence of an indefinite potential (see \cite{MR4119256,MR4081540,MR3656292}). As pointed out in \cite{MR4119256}, to overcome these difficulties and the fact that the Sobolev embedding $H^1(\mathbb{R}^3) \hookrightarrow L^2(\mathbb{R}^3)$ is not compact, Chen and Liu \cite{MR4119256} assumed that $V(x)$ was coercive to restore compactness. Then the local linking theory of Li and Willem \cite{LI19956} was applied to obtain critical points of the functional. Furthermore, Liu and Wu \cite{MR3656292} considered the case of more general $V(x)$ so that the above‑mentioned compact embedding may not hold. Therefore, stronger assumption on $f$ was imposed to restore compactness, namely,
	\[
	\lim_{|x| \to \infty} \sup_{0 < |t| \leq r} \left| \frac{f(x,t)}{t} \right| = 0, \quad \forall r > 0,
	\]
	then the  Morse theory was used to obtain critical points of the functional.
	
	In this paper, we consider the SBP system \eqref{eq:4} with an indefinite potential. Our approach is inspired by the work of Liu and Wu \cite{MR3656292} on 4-superlinear Schr\"odinger--Poisson systems.	We impose the following assumptions about the potential $V$ and the nonlinear term $f(x,u)$:
	\begin{itemize}
		\item[($V_1$)]\label{cond:v} \( V \in C(\mathbb{R}^3) \) is a bounded function  such that the quadratic form
		\[
		Q(u) = \int \left( |\nabla u|^2 + V(x)u^2 \right) dx
		\]
		is non-degenerate and the negative space of \( Q \) is finite-dimensional.
\end{itemize}
		
		Under this assumption, we may choose an equivalent norm $\|\cdot\|$ on \(X=H^1(\mathbb{R}^3)\) such that
		\[
		Q(u)=\|u^+\|^2-\|u^-\|^2,\quad u=u^++u^-,\quad u^\pm\in X^\pm,
		\]
		where \(X^+\) and \(X^-\) are the positive and negative spaces of \(Q\) respectively, $X=X^+\oplus X^-$. 
	\begin{itemize}
		\item[($V_2$)]\label{cond:v2} There exists $M_1>0$ and $R_1>0$ such that $V(x)\ge M_1$ for $|x| \ge R_1$.
		
		\item[$(f_1)$]\label{cond:f1} \( f \in C(\mathbb{R}^3 \times \mathbb{R}) \) and there exists $C>0$ such that
		\[
		|f(x,t)|\le C(|t|+|t|^5)
		\]and satisfies
		\[
		\lim_{|t| \to 0} \frac{f(x,t)}{t} = 0, \ \text{uniformly in}\ x \in \mathbb{R}^3.
		\]
		\item[$(f_2)$]\label{cond:f2} For \((x,t) \in \mathbb{R}^3 \times \mathbb{R} \setminus \{0\}\), we have \( 0 < 4F(x,t) \leq tf(x,t) \), moreover, for almost all \( x \in \mathbb{R}^3 \)
		\begin{equation}\label{eq:5}
			\lim_{|t| \to \infty} \frac{F(x,t)}{t^4} = +\infty.
		\end{equation}
		\item[$(f_3)$]\label{cond:f4} For any \( r > 0 \), we have
		\begin{equation*}
			\lim_{|x| \to \infty} \sup_{0 < |t| \leq r} \left| \frac{f(x,t)}{t} \right| = 0.
		\end{equation*}
		
		\item[$(f_4)$]\label{cond:f5} For some \(3\le s\le 5\), \( p\ge\frac{3}{2}\), and \(q\ge 6\) we have \( a \in L^\infty(\mathbb{R}^3) \cap L^p(\mathbb{R}^3), b \in L^\infty(\mathbb{R}^3) \cap L^q(\mathbb{R}^3) \) such that
		\begin{equation}\label{eq:7}
			|f(x,t)| \leq a(x)|t| + b(x)|t|^{s-1}.
		\end{equation}
	\end{itemize}
	
	The main results in this paper are as follows.
	\begin{thm}\label{th1.1}
		Suppose that $(V_1)$ and $(f_1)-(f_3)$ are satisfied, then system \eqref{eq:4} has at least one nontrivial solution.
	\end{thm}
	
	\begin{thm}\label{th1.2}
		Suppose that $(V_1)$, $(f_1)$, $(f_2)$ and $(f_4)$ are satisfied, then system \eqref{eq:4} also has at least one nontrivial solution.
	\end{thm}
	
	\begin{thm}\label{th1.4}
		Suppose that $(V_1)$, $(V_2)$ and  $(f_1)-(f_3)$ are satisfied. If $u$ is a nontrivial solution of system \eqref{eq:4}, then $u \in L^\infty(\mathbb{R}^3) \cap C^{1,\alpha}_{loc}(\mathbb{R}^3)$. Moreover, there exists positive constants $\mu$ and $C>0$ such that 
		\[
		|u(x)| \le Ce^{-\mu|x|},\ \text{for sufficiently large}\ |x|.
		\]
	\end{thm}
	\begin{remark} \label{rem:1.3}
		This estimate also holds for sign-changing solutions, which partly extends the results of Theorem 1.1 in \cite{MR4889191} where the decay behavior was only considered for positive solutions.
	\end{remark}
	\begin{thm}\label{th1.5}
		Suppose that $(V_1)$ and $(f_1)-(f_3)$ are satisfied, then system \eqref{eq:4} has a ground state solution $u_0 \in X$.
	\end{thm}
	\begin{remark} \label{rem:1.2}
		Theorem \ref{th1.5} also holds when $a=0$, which extends the result of Liu \cite{MR3656292}.
	\end{remark}
	\begin{thm}\label{th1.3}
	Suppose that $(V_1)$ and $(f_1)-(f_3)$ are satisfied,
		If $f(x,t)$ is odd, then system \eqref{eq:4} has a sequence of values $(u_n,\phi_{u_n}) \in X\times \mathcal{D} $ such that $J(u_n) \to +\infty$.
	\end{thm}
	
	The structure of this paper is arranged as follows. In Section 2, we present the notations and some preliminaries used in this paper. In Section 3, we prove that the Palais-Smale (for simplicity, we denote it as $(PS)$ in the sequel) sequence of $J$ is bounded and $J$ satisfies the $(PS)$ condition. Moreover, we use Morse theory and local linking argument to prove Theorems \ref{th1.1} and \ref{th1.2}. In Section 4, we  prove Theorem \ref{th1.4} by Moser iteration and obtain Theorem \ref{th1.5} via minimization techniques. In Section 5, we prove Theorem \ref{th1.3} by symmetric mountain pass theorem under the assumption that $f(x,t)$ is odd.
	
	\section{Preliminaries }
	In this section, we introduce some notations and preliminary results. In this paper, we use the following notation:
	\begin{itemize}
		\item Let $L^p(\mathbb{R}^3)$ ($1\le p<\infty$) be the usual Lebesgue space with norm $\|\cdot\|_p$, and $L^\infty(\mathbb{R}^3)$ the essentially bounded functions with norm $\|\cdot\|_\infty$;
		\item $L_{loc}^p(\mathbb{R}^3)$ is the space of measurable functions $u:\mathbb{R}^3 \to \mathbb{R}$ such that for every compact set $K \subset \mathbb{R}^3$, $\|u\|_{L^p(K)} <+\infty$;
		\item Let $\|\cdot\|$ be the norm of $H^1(\mathbb{R}^3)$ and denote $X=H^1(\mathbb{R}^3)$;
		\item We use the symbol $o(1)$ for a vanishing sequence in the specified space;
		\item For simplicity of writing, we write $\int f(x)$ for $\int_{\mathbb{R}^3} f(x) dx$;
		\item We use $C_1, C_2, \dots$ to denote suitable positive constants that may change from line to line.
	\end{itemize}
	
	In the study of SBP system, we also need a specific Hilbert space $\mathcal{D}$ to handle the electromagnetic potential $\phi$ arising from the second equation. Let $\mathcal{D}$ be the completion of $C_c^\infty(\mathbb{R}^3)$ with respect to the norm $\|\cdot\|_{\mathcal{D}}$ induced by the scalar product
	\[
	\langle \varphi, \psi \rangle_{\mathcal{D}} := \int \nabla \varphi \cdot \nabla \psi + a^2 \int \Delta \varphi \Delta \psi
	\]
	and $\|\cdot\|_\mathcal{D}=\sqrt{\langle \cdot, \cdot \rangle_{\mathcal{D}}}$, then $\mathcal{D}$ is a Hilbert space continuously embedded into $D^{1,2}(\mathbb{R}^3)$ and consequently in $L^6(\mathbb{R}^3)$.
	\begin{lem}\label{th2.4}({\cite[Lemma 3.3]{MR3957980}})
		For all $y\in \mathbb{R}^3, \mathcal{K}(\cdot-y)$ solves in the sense of distributions
		\[
		-\Delta \phi+a^2\Delta^2 \phi=4\pi\delta_y.
		\]
		Moreover
		\begin{enumerate}[$(i)$]
			\item if $f\in L_{loc}^1(\mathbb{R}^3)$ and, for a.e. $x\in \mathbb{R}^3$, the map $y\in \mathbb{R}^3\mapsto f(y)/|x-y|$ is summable, then $\mathcal{K}\ast f \in L_{loc}^1(\mathbb{R}^3) $;
			
			\item if $f\in L^p(\mathbb{R}^3)$ with $1\le p< {3/2}$, then $\mathcal{K}\ast f \in L^q(\mathbb{R}^3)$ for $q\in ({3p}/(3-2p),+\infty]$.
		\end{enumerate}
			In both cases $\mathcal{K}\ast f$ solves
			\[-\Delta \phi+a^2\Delta^2 \phi=4\pi f\]
			in the sense of distributions and we have the following distributional derivatives
			\[
			\nabla(\mathcal{K}\ast f)=\nabla(\mathcal{K})\ast f\quad \text{and}\quad \Delta (\mathcal{K}\ast f)=\Delta(\mathcal{K})\ast f\ \ a.e.\ \text{in}\ \mathbb{R}^3.
			\]
	\end{lem}
	By Lemma \ref{th2.4}, if we fix $u \in X$, the unique solution in $\mathcal{D}$ of the second equation in system \eqref{eq:4} is given by
	\begin{equation*}
		\phi_u:= \mathcal{K} \ast u^2.
	\end{equation*}\par
	The following lemma is important and will be used repeatedly throughout this paper.
	\begin{lem}\label{th2.3}({\cite[Lemma 3.4]{MR3957980}})
		For every $u \in X$ we have:
		\begin{enumerate}[$(i)$]
			\item for every $y \in \mathbb{R}^3$, $\phi_{u(\cdot + y)} = \phi_u(\cdot + y)$;
			
			\item $\phi_u \geq 0$;
			
			\item for every $s \in (3, +\infty]$, $\phi_u \in L^s(\mathbb{R}^3) \cap C_0(\mathbb{R}^3)$;
			
			\item for every $s \in (3/2, +\infty]$, $\nabla \phi_u = \nabla K * u^2 \in L^s(\mathbb{R}^3) \cap C_0(\mathbb{R}^3)$;
			
			\item $\phi_u \in \mathcal{D}$;
			
			\item $\|\phi_u\|_6 \leq C \|u\|^2$;
			
			\item $\phi_u$ is the unique minimizer of the functional
			\[
			E(\phi) = \frac{1}{2} \|\nabla \phi\|_2^2 + \frac{a^2}{2} \|\Delta \phi\|_2^2 - \int \phi u^2, \quad \phi \in \mathcal{D}.
			\]
			Moreover
			\item if $v_n \rightharpoonup v$ in $X$, then $\phi_{v_n} \rightharpoonup \phi_v$ in $\mathcal{D}$.
		\end{enumerate}
	\end{lem}\par
	
	\section{Proof of Theorems \ref{th1.1} and \ref{th1.2}} 
	In this section, we verify the $(PS)$ condition and compute the critical groups for the functional $J$. Under assumption $(V_1)$, the functional \eqref{f2.3} can be rewritten as
	\begin{equation*}
		J(u)=\frac{1}{2}\left(\|u^+\|^2-\|u^-\|^2\right)+\frac{1}{4}\int\phi_u u^2-\int F(x,u),
	\end{equation*}
	where $u^\pm$ is the orthogonal projection of $u$ on $X^\pm$. By the growth condition of the nonlinearity $f$ and the properties of $\phi_u$, it is easy to verify that the functional $J$ is well defined and of class $C^1$ on $H^1(\mathbb{R}^3)$. Then for any $u,v \in H^1(\mathbb{R}^3)$, the Fre\'chet derivative of $J$ is
	\begin{equation*}
		\langle J'(u),v\rangle=\int\nabla u\cdot\nabla v+\int V(x)uv+\int\phi_u uv- \int f(x,u)v.
	\end{equation*}As previously stated, solving system \eqref{eq:4} is equivalent to finding critical points of the functional $J$. Thus, to prove Theorems \ref{th1.1} and \ref{th1.2}, we first show that the functional $J$ satisfies $(PS)$ condition.
	\subsection{Palais-Smale condition}
	\begin{lem}\label{lem3.1}
		There exists $a_0>0$ such that, for every $u \in X$,
		\begin{equation}\label{f3.11}
			0\leq\int \phi_u u^2 dx \leq a_0\|u\|^4.
		\end{equation}
	\end{lem}
	
	\begin{proof}
		By the Hardy-Littlewood-Sobolev inequality \cite{LiebLoss2001}, we have
		\[
		0\leq\int \phi_u u^2  dx \leq \int \left( \frac{1}{|x|} * u^2 \right) u^2  dx \leq C \|u\|_{\frac{12}{5}}^4\leq a_0\|u\|^4.
		\]
	\end{proof}
	
	\begin{lem}\label{lem3.2}
		Suppose that $(V_1)$ and $(f_2)$ hold. Let $\{u_n\}\subset X$ be a $(PS)$ sequence of $J$, that is
		\[
		\sup_n |J(u_n)| < \infty,\qquad J'(u_n) \to 0.
		\]
		Then $\{u_n\}$ is bounded in $X$.
	\end{lem}
	\begin{proof}
		We argue by contradiction. Suppose $\{u_n\} \subset X$ is a $(PS)$ sequence, up to a subsequence, still denoted by $u_n$, we may assume \(\|u_n\| \to \infty\). Let \(v_n = \|u_n\|^{-1} u_n\), then
		\[
		v_n = v_n^+ + v_n^- \rightharpoonup v = v^+ + v^- \in X, \quad v_n^\pm, v^\pm \in X^\pm.
		\]
		
		If \(v = 0\), then \(v_n^- \to v^- = 0\) because \(\dim X^- < \infty\). Since
		\[
		\|v_n^+\|^2 + \|v_n^-\|^2 = 1,
		\]
		for \(n\) large enough we have
		\[
		\|v_n^+\|^2 - \|v_n^-\|^2 \geq \frac{1}{4}.
		\]
		Now, using \((f_2)\), we deduce that for \(n\) large enough
		\begin{align*}
			1+\sup |J(u_n)|+\|u_n\| &\ge J(u_n) - \frac{1}{4} \langle J'(u_n), u_n \rangle \\
			&= \frac{1}{4} \|u_n\|^2 \left( \|v_n^+\|^2 - \|v_n^-\|^2 \right) + \int \left( \frac{1}{4} f(x, u_n) u_n - F(x, u_n) \right) \\
			&\geq  \frac{1}{16}\|u_n\|^2,
		\end{align*}
		contradicting $\|u_n\|\rightarrow+\infty$.
		
		If $v\neq 0$, the set $A=\{x\in \mathbb{R}^3: v(x)\neq 0\}$ has positive Lebesgue measure. For $x\in A$, we have $|u_n(x)|\rightarrow+\infty$ and
		\begin{align*}
			\frac{F(x, u_n(x))}{u_n^4(x)} v_n^4(x) \to +\infty,
		\end{align*}
		thanks to \eqref{eq:5}. By the Fatou's lemma, for $x\in A$ we get
		\begin{equation}\label{f3.2}
		\int\frac{F(x, u_n)}{u_n^4} v_n^4 \ge \int_A \frac{F(x, u_n)}{u_n^4} v_n^4 \to +\infty.
		\end{equation}
		It follows from $\langle J'(u_n),u_n\rangle\rightarrow 0$ and  Lemma \ref{lem3.1} that
		\begin{align*}
			\int_A \frac{4F(x,u_n)}{u_n^4} v_n^4&\leq\int_A \frac{f(x,u_n)u_n}{u_n^4} v_n^4= \frac{1}{\|u_n\|^4} \int_A f(x,u_n)u_n \\
			&\leq \frac{1}{\|u_n\|^4} \int f(x,u_n)u_n \\
			&= \frac{1}{\|u_n\|^4} \left(\int \left(|\nabla u_n|^2 + V(x)u_n^2\right) + \int \phi_{u_n} u_n^2 - \langle J'(u_n),u_n\rangle\right) \\
			&\leq 2+ \frac{1}{\|u_n\|^4} \int \phi_{u_n} u_n^2 \\
			&\leq 2 + a_0,
		\end{align*}
		a contradiction to \eqref{f3.2}. Thus $\{u_n\}$ is bounded in $X$.
	\end{proof}
	
	To get a convergent subsequence of the $(PS)$ sequences, we need some compact properties of operators involving $\phi_u$. Consider the $C^1$-functional $\mathcal{M}: X\rightarrow\mathbb{R}$
	\[\mathcal{M}(u)=\frac{1}{4}\int \phi_u u^2,\]
	where $\phi_u=  \mathcal{K}\ast u^2=\int_{\mathbb{R}^3} \frac{1 - e^\frac{-|x-y|}{a}}{|x-y|} u^2(y)dy$. It is known that the derivative of $\mathcal{M}$ is given by
	\[\langle \mathcal{M}'(u),v\rangle=\int \phi_u uv,\]
	for all $u$, $v\in X$.

	\begin{lem}\label{lem3.3}
		The functional $\mathcal{M}$ is weakly lower semi-continuous. Moreover, $\mathcal{M}' : X \to X^*$ is weakly sequentially continuous, where $X^* = H^{-1}(\mathbb{R}^3) $ is the dual space of $X= H^{1}(\mathbb{R}^3)$.
	\end{lem}
	
	\begin{proof}
		Let \(\{u_n\}\) be a sequence in \(X\) such that \(u_n \rightharpoonup u\) in \(X\), we need to show
		\[
		\mathcal{M}(u) \leq \liminf_{n\rightarrow\infty} \mathcal{M}(u_n), \quad \langle \mathcal{M}'(u_n), v \rangle \to \langle \mathcal{M}'(u), v \rangle,
		\]
		for all \(v \in X\).
		
		Since \(u_n \rightharpoonup u\) in \(X\), up to a subsequence, still denoted by $u_n$, by the compactness of the embedding \(X \hookrightarrow L^s_{loc}(\mathbb{R}^3)\) with $s\in[2,6)$, we have
		\[
		u_n \to u \quad \text{in } L^s_{loc}(\mathbb{R}^3), \qquad u_n(x) \to u(x) \quad \text{a.e. in } \mathbb{R}^3.
		\]
		This implies that ${u_n}^2\rightarrow u^2$ a.e. in \(\mathbb{R}^3\) and $\phi_{u_n}\rightarrow \phi_u$ a.e. in \(\mathbb{R}^3\). Therefore, by the properties of $\phi_u\geq 0$ and the Fatou's lemma
		\[
		\mathcal{M}(u) = \frac{1}{4} \int  \phi_uu^2\leq \frac{1}{4} \liminf_{n\rightarrow\infty} \int\phi_{u_n}u_n^2 = \liminf_{n\rightarrow\infty} \mathcal{M}(u_n).
		\]
		Hence \(\mathcal{M}\) is weakly lower semi-continuous.
		
		To prove the weak sequential continuity of \(\mathcal{M}'\), we will prove that
		\begin{align*}
			\langle \mathcal{M}'(u_n) - \mathcal{M}'(u), v \rangle
			&= \int \left(\phi_{u_n}u_n v-\phi_{u}u v \right)\notag\\
			&=\int \phi_{u_n}(u_n-u)v+\int (\phi_{u_n}-\phi_u)uv\notag\\
			&=:I_1+I_2\rightarrow 0.
		\end{align*}
		By Lemmas \ref{th2.3} and \ref{lem3.1},  H\"older inequality and Sobolev embedding inequalities, we have
		\begin{align*}
			\int |\phi_{u_n}(u_n-u)|^2
			&\leq \|\phi_{u_n}\|_6^2 \|u_n - u\|_3^2 \notag\\
			&\leq C \|u_n\|^4 \|u_n - u\|^2 \leq C.
		\end{align*}
		Combining \(u_n \to u\) a.e. in \(\mathbb{R}^3\) and \(\phi_{u_n} \to \phi_u\) a.e. in \(\mathbb{R}^3\), we have \(\phi_{u_n}(u_n - u) \rightharpoonup 0\) in \(L^2(\mathbb{R}^3)\). Thus
		\begin{align}\label{I1}
			I_1=\int \phi_{u_n} (u_n - u) v \to 0.
		\end{align}
		Since $\phi_{u_n}\rightharpoonup \phi_{u}$ in $L^6(\mathbb{R}^3)$ and $uv \in L^\frac{6}{5}(\mathbb{R}^3)$, then
		\begin{align}\label{I2}
			I_2=\int \left( \phi_{u_n} - \phi_u \right) u v \to 0.
		\end{align}
		From \eqref{I1} and \eqref{I2}, for \(v \in X\), we have
		\[
		\langle \mathcal{M}'(u_n), v \rangle \to \langle \mathcal{M}'(u), v \rangle.
		\]
		Therefore, the proof of the lemma is complete.
	\end{proof}
	
	\begin{lem}\label{lem3.4}
		Let \( u_n \rightharpoonup u \) in \( X \), then
		\[
		\liminf_{n\rightarrow\infty}\int \phi_{u_n} u_n (u_n - u) \geq 0.
		\]
	\end{lem}
	
	\begin{proof}
		Applying Lemma \ref{lem3.3}, we have
		\[
		\liminf_{n\rightarrow\infty}\mathcal{M}(u_n) \geq \mathcal{M}(u), \quad \lim_{n\rightarrow\infty} \langle \mathcal{M}'(u_n), u \rangle = \langle \mathcal{M}'(u), u \rangle.
		\]
		Therefore
		\begin{align*}
			\liminf_{n\rightarrow\infty} \int \phi_{u_n} u_n (u_n - u)
			&= \liminf_{n\rightarrow\infty} \left( \int \phi_{u_n} u_n^2 - \int \phi_{u_n} u_n u \right) \\
			&= \liminf_{n\rightarrow\infty}\left( 4\mathcal{M}(u_n) - \langle \mathcal{M}'(u_n), u \rangle \right) \\
			&\geq 4\mathcal{M}(u) - \langle \mathcal{M}'(u), u \rangle = 0.
		\end{align*}
	\end{proof}
	
	\begin{lem}\label{lem3.5}
		Assume $(V_1)$, $(f_2)$ and $(f_3)$ are satisfied, the functional $J$ satisfies the (PS) condition.
	\end{lem}
	
	\begin{proof}
		Suppose $\{u_n\}\subset X$ is a $(PS)$ sequence. By Lemma \ref{lem3.2}, then $\{u_n\}$ is bounded in $X$. Up to a subsequence, still denoted by $u_n$, we may assume \( u_n \rightharpoonup u \) in \( X \). We have
		\[
		\int \left( \nabla u_n \cdot \nabla  u+ V(x) u_n u \right)\rightarrow\int \left( |\nabla u|^2 + V(x) u^2 \right) = \|u^+\|^2 - \|u^-\|^2.
		\]
		Consequently
		\begin{align*}
			o(1) &= \langle J'(u_n), u_n - u \rangle \\
			&= \int \big[ \nabla u_n \cdot \nabla (u_n - u) + V(x) u_n (u_n - u) \big] \\
			&\quad + \int \phi_{u_n} u_n (u_n - u) - \int f(x, u_n) (u_n - u) \\
			&=\int\left(\nabla u_n \cdot \nabla  u_n+ V(x) u_n^2\right)-\int\left( \nabla u_n \cdot \nabla  u+ V(x) u_n u \right)\\
			&\quad + \int \phi_{u_n} u_n (u_n - u) - \int f(x, u_n) (u_n - u) \\
			&= \left( \|u_n^+\|^2 - \|u_n^-\|^2 \right) - \left( \|u^+\|^2 - \|u^-\|^2 \right) \\
			&\quad + \int \phi_{u_n} u_n (u_n - u) - \int f(x, u_n) (u_n - u) + o(1).
		\end{align*}
		Because $\dim X^-<+\infty$, we have $\|u_n^-\|\rightarrow \|u^-\|$. The above equality
		\begin{align*}
			\|u_n^+\|^2 - \|u^+\|^2 =\int f(x, u_n) (u_n - u)-\int \phi_{u_n} u_n (u_n - u)  + o(1).
		\end{align*}
		With assumption $(f_3)$, the following inequality was proved in \cite[P.29]{Bartsch2004}
		\begin{align*}\label{s2}
			\limsup_{n\rightarrow\infty}\int f(x, u_n) (u_n - u)\leq 0.
		\end{align*}
		Using Lemma \ref{lem3.4}, we obtain
		\begin{align*}
			\limsup_{n\rightarrow\infty}\left(  \|u_n^+\|^2 - \|u^+\|^2\right)&=\limsup_{n\rightarrow\infty}\left\{\int f(x, u_n) (u_n - u)-\int \phi_{u_n} u_n (u_n - u) \right\}\\
			&=\limsup_{n\rightarrow\infty}\int f(x, u_n) (u_n - u)-\liminf_{n\rightarrow\infty}\int\phi_{u_n} u_n (u_n - u)\\
			&\leq 0.
		\end{align*}
		Together this with the weakly lower semi-continuity of the norm functional $u\mapsto\|u\|$, we have
		\[\|u^+\|^2\leq\liminf_{n\rightarrow \infty}\|u_n^+\|^2\leq\limsup_{n\rightarrow \infty}\|u_n^+\|^2\leq\|u^+\|^2. \]
		Thus $\|u_n^+\|\rightarrow\| u^+\|$. Combining with $\|u_n^-\|\rightarrow\| u^-\|$, we get $\|u_n\|\rightarrow\| u\|$. Together with $u_n\rightharpoonup u$ in $X$, we have $u_n\rightarrow u$ in $X$. The proof of the lemma is complete.
	\end{proof}
	\subsection{Existence of nontrivial solutions}
	In the previous part, we have proved that $J$ satisfies the $(PS)$ condition. In this part, we will prove Theorems \ref{th1.1} and \ref{th1.2}. Here, we will recall some concepts and results from infinite-dimensional Morse theory, see Chang \cite{MR1196690} and Mawhin-Willem \cite[Chapter 8]{MR982267}. Let $E$ be a Banach space, $\Phi: E\rightarrow\mathbb{R}$ be a $C^1$-functional, $u$ is an isolated critical point and $\Phi(u)=c$. Then the $i$-th critical group of $\Phi$ at $u$ defined by
	\[C_i(\Phi, u):=H_i(\Phi_c,\Phi_c\backslash\{0\})\text{,}  \quad i\in\mathbb{N}=\{0,1,2,\cdots\}\text{,} \]
	where $\Phi_c:=\Phi^{-1}(-\infty, c]$ and $H_*$ stands for the singular homology with coefficients in $\mathbb{Z}$.
	
	If $\Phi$ satisfies the $(PS)$ condition and the critical values of $\Phi$ are bounded from below by $\alpha$, then following Bartsch-Li \cite{MR1420790}, we define the $i$-th critical group of $\Phi$ at infinity by
	\[C_i(\Phi,\infty):=H_i(E, \Phi_\alpha)\text{,} \quad i\in\mathbb{N}\text{.} \]
	
	\begin{prop}[{\cite[Proposition 3.6]{MR1420790}}]
		\label{pro2}
		If $\Phi\in C^1(E,\mathbb{R})$ satisfies the (PS) condition and $C_k(\Phi,0)\neq C_k(\Phi,\infty)$ for some $k\in\mathbb{N}$, then $\Phi$ has a nonzero critical point.
	\end{prop}
	
	\begin{prop}[{\cite[Theorem 2.1]{MR1110119}}] \label{pro3}
		Suppose $\Phi\in C^1(E,\mathbb{R})$ has a local linking at $0$ with respect to the decomposition $E=Y\oplus Z$, i.e., for some $\rho>0$,
		\begin{align*}
			&\Phi(u)\leq 0~~\text{\quad for}~~u\in Y\cap B_\rho\text{,} \\
			&\Phi(u)>0~~\text{\quad for}~~u\in (Z\backslash\{0\})\cap B_\rho,
		\end{align*}
		where $B_\rho=\{u\in E :\|u\|\leq \rho\}$. If $l=\dim Y<\infty$, then $C_l(\Phi, 0)\neq 0$.
	\end{prop}
	
	To prove Theorems \ref{th1.1} and \ref{th1.2}, given that the functional $J$ is assumed to have only finitely many critical points and satisfies the $(PS)$ condition, its critical groups at infinity are well-defined. To establish the critical group $C_*(J,\infty)$ of $J$ at infinity, we need the following lemma.
	
	\begin{lem}\label{lem4}
		There exists \( A > 0 \) such that, if \( J(u) \le -A \), then
		\[
		\left.\frac{\mathrm{d}}{\mathrm{d}t} \right|_{t=1}J(tu) < 0.
		\]
	\end{lem}
	
	\begin{proof}
		We argue by contradiction and assume that there exists a sequence \(\{u_{n}\}\subset X\) such that \(J(u_{n})\leq -n\) but
		\begin{align}\label{s4}
			\langle J^{\prime}(u_{n}),u_{n}\rangle = \left.\frac{\mathrm{d}}{\mathrm{d}t} \right|_{t=1}J(tu_{n}) \geq 0.
		\end{align}
		Thus, by $(f_2)$, we have
		\begin{align}\label{s3}
			\|u_{n}^{+}\|^{2} - \|u_{n}^{-}\|^{2}
			&\leq \left( \|u_{n}^{+}\|^{2} - \|u_{n}^{-}\|^{2} \right) + \int [f(x,u_{n})u_{n} - 4F(x,u_{n})] \notag\\
			&= 4J(u_{n}) - \langle J^{\prime}(u_{n}),u_{n}\rangle \leq -4n.
		\end{align}
		Let \(v_{n} = \|u_{n}\|^{-1}u_{n}\) and \(v_{n}^{\pm}\) be the orthogonal projection of \(v_{n}\) on \(X^{\pm}\). Then \(v_{n}^{-} \to v^{-}\) for some \(v^{-} \in X^{-},\) because \(\dim X^{-} < \infty.\)
		
		If \(v^{-} \neq 0\), then there exists some \(v \in X \backslash \{0\}\) such that \(v_{n} \rightharpoonup v\) in \(X\). Similar to \eqref{f3.2}, we have
		\[
		\begin{aligned}
			\int \frac{f(x,u_{n})u_{n}}{\|u_{n}\|^{4}}
			&\geq 4 \int \frac{F(x,u_{n})}{\|u_{n}\|^{4}} \\
			&\geq 4 \int_{v \neq 0} \frac{F(x,u_{n})}{u_{n}^{4}} v_{n}^{4} \to +\infty.
		\end{aligned}
		\]
		By \eqref{s4}, we get
		\[
		\begin{aligned}
			0 &\leq \frac{\langle J'(u_{n}), u_{n}\rangle}{\|u_{n}\|^{4}} \\
			&= \frac{\|u_{n}^{+}\|^{2} - \|u_{n}^{-}\|^{2}}{\|u_{n}\|^{4}} + \frac{1}{\|u_{n}\|^{4}} \int \phi_{u_{n}} u_{n}^{2} \,  - \int \frac{f(x,u_{n})u_{n}}{\|u_{n}\|^{4}} \,  \\
			&\leq 1 + a_{0} - \int \frac{f(x,u_{n})u_{n}}{\|u_{n}\|^{4}} \, \to -\infty,
		\end{aligned}
		\]
		a contradiction.
		
		Therefore \(v^{-} = 0\). From
		\[
		\|v_{n}^{+}\|^{2} + \|v_{n}^{-}\|^{2} = 1,
		\]
		we have \(\|v_{n}^{+}\| \to 1.\) Consequently, for \(n\) large enough,
		\[
		\|u_{n}^{+}\| = \|u_{n}\| \|v_{n}^{+}\| \geq \|u_{n}\| \|v_{n}^{-}\| = \|u_{n}^{-}\|,
		\]
		violating \eqref{s3}. Hence the proof of the lemma is complete.
	\end{proof}
	
	\begin{lem}\label{group1}
		$C_l(J,\infty)=0$ for all $l=0,1,2,\cdots$.
	\end{lem}
	\begin{proof}
		Let \( B = \{ v \in X \mid \|v\| \leq 1 \} \), \( S = \partial B \) be the unit sphere in \( X \), and \( A > 0 \) be the number given in Lemma \ref{lem4}. Without loss of generality, we may assume that
		\[
		-A < \inf_{\|u\| \leq 2} J(u).
		\]
		By \eqref{eq:5}, it is clear that for any \( v \in S \),
		\[
		\begin{aligned}
			J(sv)
			&= \frac{s^2}{2} \left( \|v^+\|^2 - \|v^-\|^2 \right) + \frac{s^4}{4} \int \phi_v v^2 - \int F(x, sv) \\
			&= s^4 \left\{ \frac{\|v^+\|^2 - \|v^-\|^2}{2s^2} + \frac{1}{4} \int \phi_v v^2 - \int \frac{F(x, sv)}{s^4} \right\} \to -\infty,
		\end{aligned}
		\]
		as \( s \to +\infty \). So there exists \( s_v > 0 \) such that \( J(s_v v) = -A \). Let \( u = s_v v \). Then, by a direct computation together with Lemma \ref{lem4}, we obtain
		\[
		\left. \frac{d}{ds} \right|_{s=s_v} J(sv) = \frac{1}{s_v} \left. \frac{d}{dt} \right|_{t=1} J(tu) < 0.
		\]
		By the implicit function theorem, then \( T : v \mapsto s_v \) is a continuous function on \( S \). Using this map we can construct a deformation from $\eta :X\backslash B \to J_{-A}$
		\[
		\eta(u) =
		\begin{cases}
			u, & \text{if } J(u) \leq -A, \\
			T\left( \frac{u}{\|u\|} \right) \frac{u}{\|u\|}, & \text{if } J(u) > -A.
		\end{cases}
		\]
		Thus, we deduce via the homotopy invariance of singular homology
		\[C_l(J,\infty)=H_l(X, J_{-A})\cong H_l(X, X\backslash B)=0\text{,\quad}\text{for}~~ l\in\mathbb{N}\text{.}\]
		The proof of the lemma is complete.
	\end{proof}
	\noindent
	\textbf{Proof of Theorem \ref{th1.1}}
	We have proved that $J$ satisfies the $(PS)$ condition thanks to Lemma \ref{lem3.5}. Next, we claim $J$ has a local linking at $0$ with respect to the decomposition $X=X^+\oplus X^-$. In fact, by $(f_1)$, for any $\varepsilon>0$, there is $C_\varepsilon>0$ such that
	\begin{align}\label{f3.01}
		|f(x,u)|\leq \varepsilon |u|+C_\varepsilon |u|^5.
	\end{align}
	Then, by  Sobolev embedding inequality, integrating yields
	\[
	\int F(x,u)\le \frac{\varepsilon}{2}\|u\|_2^2+\frac{C_\varepsilon}{6} \|u\|_6^6 \le \frac{\varepsilon}{2}\|u\|^2+\frac{C_\varepsilon}{6} \|u\|^6.
	\]
	Together with \eqref{f3.11}, we obtain
	\[J(u)=\frac{1}{2}\left(\|u^+\|^2-\|u^-\|^2\right)+o(\|u\|^2)\text{\quad as\quad}\|u\|\rightarrow 0.\]
	There exists $c>0$ and if $u\in (X^+\setminus\{0\})\cap B_c $, we have
	\[J(u)=\frac{1}{2}\left(\|u^+\|^2-\|u^-\|^2\right)+o(\|u\|^2)\ge\frac{1}{2}\|u^+\|^2+o(\|u\|^2)>0.\]
	Similarly, if $u \in X^-\cap B_c$, then we obtain
	\[J(u)=\frac{1}{2}\left(\|u^+\|^2-\|u^-\|^2\right)+o(\|u\|^2)\le-\frac{1}{2}\|u^-\|^2+o(\|u\|^2)\le 0.\]
	Hence $J$ has a local linking at $0$ with respect to the decomposition $X=X^+\oplus X^-$. Together with Proposition \ref{pro3}, we obtain
	\[C_l(J,0)\neq 0\]
	for $l=\dim X^-$. From Lemma \ref{group1}, we have
	\[C_l(J, \infty)\neq C_l(J, 0)\text{.}\]
	This concludes, by Proposition \ref{pro2}, that $J$ has a nontrivial critical point $u$.  
	Thus, $u$ is a nontrivial solution of \eqref{eq:4}.\qed
	
	We prove Theorem \ref{th1.2} by employing a similar approach to that used in the literature \cite{Jiang2023}, we give the details below.
	\begin{lem}\label{lem5.1}
		If $f:\mathbb{R}^3\times\mathbb{R}\to\mathbb{R}$ is continuous and satisfies $(f_4)$. Then the functional $K:X\to\mathbb{R}$
		\[
		K(u)=\int F(x,u)
		\]
		is well defined and of class $C^1$ with
		\[
		\langle K'(u),\phi\rangle=\int f(x,u)\phi,\quad \forall\,\phi\in X.
		\]
		Therefore, $K'$ is compact.
	\end{lem}
	
	\begin{proof}
		From \eqref{eq:7} we have
		\[
		|f(x,t)|\leq \|a\|_\infty|t|+\|b\|_\infty|t|^{s-1},
		\]
		then
		\[
		|F(x,t)|\leq C_1\|a\|_\infty|t|^2+C_2\|b\|_\infty|t|^s.
		\]
		Since $s \in [3,5]$, it is well known that $K$ is well defined and of class $C^1$.
		
		To prove that $K':X\to X^*$ is compact, let $u_n\rightharpoonup u$ in $X$. Because $a\in L^p(\mathbb{R}^3),b\in L^q(\mathbb{R}^3)$, for any $\varepsilon>0$, there exists $R>0$ such that
		\begin{equation}\label{5.1}
			\int_{B_R^c}|a|^p<\varepsilon^p,\quad \int_{B_R^c}|b|^q<\varepsilon^q,
		\end{equation}
		where $B_R$ represents the ball centered at the origin with radius $R>0$ in $\mathbb{R}^3$, and $B_R^c=\mathbb{R}^3\setminus B_R$ is the complement set. For $\phi\in X$, let $\|\phi\|=1$, by invoking \eqref{eq:7}, \eqref{5.1} and H\"older inequality, we have
		\begin{equation*}
			\begin{aligned}
				\int_{B_R^c}|f(x,u)|\,|\phi|
				&\leq \int_{B_R^c}|a||u||\phi|+\int_{B_R^c}|b||u|^{s-1}|\phi| \\
				&\leq C\|a\|_{p}\|u\|_{2p/(p-1)}\|\phi\|_{2p/(p-1)}+C\|b\|_q\|u\|_{sq/(q-1)}^{s-1}\|\phi\|_{sq/(q-1)} \\
				&\leq \varepsilon \|u\|_{2p/(p-1)}\|\phi\|_{2p/(p-1)}+\varepsilon \|u\|_{sq/(q-1)}^{s-1}\|\phi\|_{sq/(q-1)} \\
				&\leq \varepsilon C_1\|u\|+\varepsilon C_2\|u\|^{s-1}\\
				&\leq M\varepsilon,
			\end{aligned}
		\end{equation*}
		where $\|\,\cdot\,\|_\gamma$ is the standard $L^\gamma(B_R^c)$ norm, $M$ is a constant that depends on $\sup\limits_n\|u_n\|$ but not on $\phi$.
		
		Similarly, 
		\[
		\int_{B_R^c}|f(x,u_n)|\,|\phi|\leq M\varepsilon. 
		\]
		Therefore
		\begin{equation}\label{5.3}
			\int_{B_R^c}\big|f(x,u_n)-f(x,u)\big|\,|\phi|
			\leq \int_{B_R^c}|f(x,u_n)|\,|\phi|+\int_{B_R^c}|f(x,u)|\,|\phi|\leq 2M\varepsilon.
		\end{equation}
		By the compactness of the embedding $X\hookrightarrow L^r_{loc}(\mathbb{R}^3)$, $r \in[2,6)$, we have
		\begin{equation}\label{5.4}
			\sup_{ \|\phi\|=1}\int_{B_R}\big|f(x,u_n)-f(x,u)\big|\,|\phi|\to 0.
		\end{equation}
		It follows from \eqref{5.3} and \eqref{5.4} that
		\[
		\begin{aligned}
			\varlimsup_{n\to\infty}\|K'(u_n)-K'(u)\|
			&= \varlimsup_{n\to\infty}\sup_{\|\phi\|=1}\left|\int\big(f(x,u_n)-f(x,u)\big)\phi\right| \\
			&\leq \varlimsup_{n\to\infty}\sup_{\|\phi\|=1}\Bigg(
			\int_{B_R}\big|f(x,u_n)-f(x,u)\big|\,|\phi| \\
			&\quad + \int_{B_R^c}\big|f(x,u_n)-f(x,u)\big|\,|\phi|\Bigg) \\
			&\leq 2M\varepsilon.
		\end{aligned}
		\]
		Let $\varepsilon\to 0$, we deduce $K'(u_n)\to K'(u)$ in $X^*$. Therefore $K'$ is compact. The proof of the lemma is complete.
	\end{proof}
	\noindent
	\textbf{Proof of Theorem \ref{th1.2}} According to the proof of Lemma \ref{lem3.5}, we need to obtain
	\[
	\varlimsup_{n\to\infty}\int f(x,u_n)(u_n-u)=0,
	\]
	from $u_n\rightharpoonup u$ in $X$.
	
	Indeed, if $u_n\rightharpoonup u$ in $X$, by Lemma \ref{lem5.1}, we have $K'(u_n)\to K'(u)$ and
	\[
	\begin{aligned}
		\left|\int f(x,u_n)(u_n-u)\right|
		&= \bigl|\langle K'(u_n),u_n-u\rangle\bigr| \\
		&\leq \bigl|\langle K'(u_n)-K'(u),u_n-u\rangle\bigr|+\bigl|\langle K'(u),u_n-u\rangle\bigr| \\
		&\leq \|K'(u_n)-K'(u)\|\,\|u_n-u\|+o(1)\to 0.
	\end{aligned}
	\]
	Therefore, similar to the proof of Lemma \ref{lem3.5}, we can show that the functional $J$ satisfies the $(PS)$ condition. Moreover, by invoking Lemmas \ref{lem4} and \ref{group1}, using the same proof we conclude that under the assumptions of Theorem \ref{th1.2}, $J$ has a nontrivial critical point $u$. Thus, $u$ is a nontrivial solution of \eqref{eq:4}. \qed
	
	\section{Proof of Theorems \ref{th1.4} and \ref{th1.5}}
	In this section, let $u$ denote the solution of system \eqref{eq:4} which we have obtained in the preceding section. We will consider the regularity and exponential decay of $u$, and thus complete the proof of Theorem \ref{th1.4}. Lastly, we obtain the existence of a ground state solution of the system \eqref{eq:4} to complete the proof of Theorem \ref{th1.5}.
	\subsection{Exponential decay}
	\begin{lem}\label{lem5.11}
		Under conditions $(V_1)$ and $(f_1)$, if $u$ is a nontrivial solution of \eqref{eq:4}, then $u \in C^{1,\alpha}_{loc}(\mathbb{R}^3)$.
	\end{lem}
	\begin{proof}
		For any $r>0$, let $B_1=B_r(0)$ and $B_2=B_{2r}(0)$. Obviously, $B_1 \subset\subset B_2$.
		Noting that $u$ is also a nontrivial solution of the following:
		\[
		- \Delta u=f(x,u)-V(x)u-\phi_u u:=g(u).
		\]
		Let
		\[
		a(x):= \frac{g(u)}{u}= \frac{f(x,u)}{u}-V(x)-\phi_u.	\]
		From  $(V_1)$, $(f_1)$ and H\"older inequality, it is clear to see that 
		\[
		\int_{B_2}|a(x)|^\frac{3}{2}< +\infty,
		\]
		that is
		\[
		a(x) \in L^\frac{3}{2}(B_2).
		\]
		It follows from the Br\'ezis-Kato Theorem \cite{Brezis1978} that $u \in L^q(B_2)$ for any $1 \le q <\infty$. From Lemma \ref{th2.3}, we have $\phi_u \in L^\infty$. Since $(V_1)$, $(f_1)$, $\phi_u \in L^\infty$ and $u \in L^q(B_2)$ for any $1 \le q <\infty$, we have
		\begin{align*}
			\int_{B_2}|a(x)u|^q&=\int_{B_2}|f(x,u)-V(x)u-\phi_{u}u|^q\\
			&\leq\int_{B_2}|u|^q+C_1\int_{B_2}|u|^{5q}+C_2\int_{B_2}|u|^{q}+C_3\int_{B_2}|u|^{q}< +\infty,
		\end{align*}
		that is
		\[
		a(x)u=f(x,u)-V(x)u-\phi_{u}u \in L^q(B_2).
		\]
		We have
		\[
		-\Delta u=a(x)u \in L^q(B_2).
		\]
		By the Calderon-Zygmund inequality and Schauder estimates, see Theorem 9.9 of \cite{Gilbarg2001}, we have $u \in W^{2,q}(B_2)$. It follows from classic interior $L^q$-estimates that
		\[
		\|u\|_{W^{2,q}(B_1)} \le C(\|a(x)u\|_{L^q(B_2)}+\|u\|_{L^q(B_2)}),
		\]
		where $C=C(r,q)$. For any $ q \ge 2$, by the Sobolev embedding theorem we get $u \in C^{1,\alpha}(\bar{B_1})$ for some $\alpha \in (0,1)$. Consequently, we deduce that $u \in L^\infty(B_1)$ and $u \in C^{1,\alpha}_{loc}(\mathbb{R}^3)$ with $\alpha \in (0,1)$.\par
	\end{proof}
	\begin{lem}\label{lem5.12}
		Under conditions $(V_2)$ and $(f_1)$, if $u$ is a nontrivial solution of \eqref{eq:4}, then 
		\[ |u(x)| \to 0 \quad \text{as} \quad |x| \to +\infty. \]
	\end{lem}
	\begin{proof}
		The proof is similar to the proof of Theorem 1.11 in \cite{1990Some}. For completeness, we give the
		proof of this Lemma. For any $R>2R_1>0$, $0<r<\frac{R}{2}$, let \(\eta \in C^\infty(\mathbb{R}^3)\), \(0 \leq \eta \leq 1\) with
		\[
		\eta(x) =
		\begin{cases}
			1 & \text{if } |x| \geq R, \\[4pt]
			0 & \text{if } |x| < R - r,
		\end{cases}
		\quad \text{and} \quad
		|\nabla \eta| \leq \frac{2}{r}.
		\]
		We set $u^+=\max\{0,u\}$, $u_L^+=\min\{u^+,L\}$, where $L>0$. To this end, we suppose that
		\[
		v=\eta^2 u^+ (u_L^+)^{2(\gamma-1)} \quad \text{and} \quad W_L=\eta u^+(u_L^+)^{\gamma-1},
		\]
		for any $\gamma \ge 1$. Let $v$ be the test function, multiplying both sides of the equation by $v$ and integrating over $\mathbb{R}^3$, we have
		\[
		\int\big(\nabla u\cdot \nabla v+ V(x)uv+\phi_u uv\big)=\int f(x,u)v
		\]
		By \eqref{f3.01}, let $\varepsilon=\frac{M_1}{2}$ and $C_\varepsilon=C$, one has
		\begin{equation}\label{7.1}
			\begin{aligned}
				\int f(x,u)v &= \int f(x,u)\eta^2 u^+ (u_L^+)^{2(\gamma-1)} \\
				&\le\frac{M_1}{2}\int\eta^2 (u^+)^2 (u_L^+)^{2(\gamma-1)}+C\int\eta^2 (u^+)^6 (u_L^+)^{2(\gamma-1)}.
			\end{aligned}
		\end{equation}
		By direct calculation, we obtain
		\begin{equation}\label{7.3}
			\begin{aligned}
				\int \nabla u\cdot \nabla v &= 2\int\nabla u^+\cdot \nabla\eta \eta u^+(u_L^+)^{2(\gamma-1)}+\int|\nabla u^+|^2\eta^2(u_L^+)^{2(\gamma-1)} \\
				&\quad +2(\gamma-1)\int\nabla u^+\cdot\nabla u_L^+\eta^2u^+(u_L^+)^{2(\gamma-1)-1}\\
				&\ge 2\int\nabla u^+\cdot \nabla\eta \eta u^+(u_L^+)^{2(\gamma-1)}+\int|\nabla u^+|^2\eta^2(u_L^+)^{2(\gamma-1)}.
			\end{aligned}
		\end{equation}
		It follows from \eqref{7.1}, \eqref{7.3} and $(V_2)$ that,
		\begin{equation*}
			\begin{aligned}
				\int|\nabla u^+|^2\eta^2(u_L^+)^{2(\gamma-1)}
				&\le -2\int\nabla u^+\cdot \nabla\eta \eta u^+(u_L^+)^{2(\gamma-1)}+\int \nabla u\cdot \nabla v \\
				&=-2\int\nabla u^+\cdot \nabla\eta \eta u^+(u_L^+)^{2(\gamma-1)}+\int (f(x,u)v-V(x)uv-\phi_u uv)\\
				&\le -2\int\nabla u^+\cdot \nabla\eta \eta u^+(u_L^+)^{2(\gamma-1)}+\frac{M_1}{2}\int\eta^2 (u^+)^2 (u_L^+)^{2(\gamma-1)}\\
				&\quad +C\int\eta^2 (u^+)^6 (u_L^+)^{2(\gamma-1)}-M_1\int\eta^2 (u^+)^2 (u_L^+)^{2(\gamma-1)}.
			\end{aligned}
		\end{equation*}
		Applying the Young's inequality to the cross term, we can obtain
		\begin{equation*}
			\begin{aligned}
				\int|\nabla u^+|^2\eta^2(u_L^+)^{2(\gamma-1)}
				&\le \frac{1}{2}\int|\nabla u^+|^2\eta^2(u_L^+)^{2(\gamma-1)}+C\int|\nabla\eta|^2 (u^+)^2 (u_L^+)^{2(\gamma-1)}\\
				&\quad +C\int\eta^2 (u^+)^6 (u_L^+)^{2(\gamma-1)}.
			\end{aligned}
		\end{equation*}
		This implies that
		\begin{equation}\label{7.4}
			\int|\nabla u^+|^2\eta^2(u_L^+)^{2(\gamma-1)}
			\le C\int|\nabla\eta|^2 (u^+)^2 (u_L^+)^{2(\gamma-1)}+C\int\eta^2 (u^+)^6 (u_L^+)^{2(\gamma-1)}.
		\end{equation}
		According to the definition of $W_L$,
		combining \eqref{7.4} and Sobolev embedding inequality, let $\gamma=3$, we have,
		\begin{equation}\label{7.6}
			\begin{aligned}
				\|W_L\|_{2\gamma}^2 &\le C \int |\nabla W_L|^2 dx \\
				&\le C \left\{ \int |\nabla \eta|^2 (u^+)^2 (u_L^+)^{2(\gamma-1)} + \gamma^2\int (\nabla u^+)^2 \eta^2 (u_L^+)^{2(\gamma-1)} \right\} \\
				&\le C\gamma^2\left\{\int|\nabla\eta|^2 (u^+)^2 (u_L^+)^{2(\gamma-1)}+\int\eta^2 (u^+)^6 (u_L^+)^{2(\gamma-1)}\right\}.\\
			\end{aligned}
		\end{equation}
		We claim that
		\[
		u \in L^{18}(|x| \ge R)
		\]
		for \( R \) large enough. In fact, from \eqref{7.6} we have
		\begin{equation*}
			\begin{aligned}
				\left\{\int(\eta u^+(u_L^+)^2)^6\right\}^\frac{1}{3}
				&\le 9C \left\{ \int (\eta u^+ (u_L^+)^2)^6\right\}^\frac{1}{3} \left\{ \int_{|x|>R-r} (u^+)^6\right\}^\frac{2}{3}\\
				&+9C\int |\nabla \eta|^2 (u^+)^2 (u_L^+)^4. \\
			\end{aligned}
		\end{equation*}
		Since $u^+ \in L^6(\mathbb{R}^3)$, $|u^+|_{L^6(|x|>R)} <\frac{1}{9C}$ for \( R \) large enough.
		Hence, we get
		\begin{align*}
			\left\{ \int_{|x|\ge R} \left( u^+ (u_L^+)^2 \right)^{6} \right\}^{\frac{1}{3}}
			&\le \left\{ \int \left( \eta u^+ (u_L^+)^2 \right)^{6} \right\}^{\frac{1}{3}} \\
			&\le \int |\nabla \eta|^2 (u^+)^2 (u_L^+)^4 \\
			&\le \frac{C}{r^2} \int (u^+)^6 <+\infty,
		\end{align*}
		which implies that \( u \in L^{18}(|x| \ge R) \). Next, we let \( \gamma = \frac{3(t-1)}{t} \) with \( t = \frac{9}{2} \) and \( u^+ \in L^{\frac{2\gamma t}{t-1}}(|x| \ge R-r) \). Then by \eqref{7.6} we have
		\begin{align*}
			\|W_L\|_6^2 &=\|W_L\|_{\frac{2\gamma t}{t-1}}^2\\
			&\le C \gamma^2 \left( \int_{|x|\ge R-r} \left( \eta^2 (u^+)^{2\gamma} \right)^{\frac{t}{t-1}} dx \right)^{1-\frac{1}{t}} \left( \int_{|x|\ge R-r} (u^+)^{4t} dx \right)^{\frac{1}{t}} \\
			&\quad + C \gamma^2 \frac{[R^3 - (R-r)^3]}{r^2}^\frac{1}{t} \left( \int_{|x|\ge R-r} (u^+)^{{2\gamma} \frac{t}{t-1}} dx \right)^{1-\frac{1}{t}}\\
			&\le C \gamma^2 \left(1+\frac{R^\frac{3}{t}}{r^2}\right)\left( \int_{|x|\ge R-r} (u^+)^{{2\gamma} \frac{t}{t-1}} dx \right)^{1-\frac{1}{t}}.
		\end{align*}
		Let \( L \to +\infty \) in above inequality, then, we obtain
		\[
		\|u^+\|_{L^{6\gamma}(|x|\ge R)}^{2\gamma} \le C \gamma^2 \left(1 + \frac{R^{\frac{3}{t}}}{r^2}\right) \left(\int_{|x|\ge R-r} (u^+)^{\frac{2\gamma t}{t-1}} dx\right)^{1-\frac{1}{t}}.
		\]
		Let \( \chi = \frac{3(t-1)}{t} \), \( s = \frac{2t}{t-1} \), then
		
		\begin{equation}\label{7.7}
			\|u^+\|_{L^{\gamma\chi s}(|x|\ge R)} \le C^{\frac{1}{2\gamma}} \gamma^{\frac{1}{\gamma}} \left(1 + \frac{R^{\frac{3}{t}}}{r^2}\right)^{\frac{1}{2\gamma}} \|u^+\|_{L^{\gamma s}(|x|\ge R-r)}.
		\end{equation}
		If we set \( \gamma = \chi^m \) (\( m = 1,2,\dots \)), from \eqref{7.7} then we get
		\begin{equation}\label{7.8}
			\|u^+\|_{L^{\chi^{m+1} s}(|x|\ge R)} \le C^{\frac{1}{2\chi^m}} \chi^{m\chi^{-m}} \left(1 + \frac{R^{\frac{3}{t}}}{r^2}\right)^{\frac{1}{2\chi^m}} \|u^+\|_{L^{\chi^m s}(|x|\ge R-r)}.
		\end{equation}
		It is clear that \( 2 > \frac{N}{t} \). Let \( r_m = 2^{-(m+1)}R \), by \eqref{7.8} and the Moser iteration implies
		\begin{align*}
			\|u^+\|_{L^{\chi^{m+1} s}(|x|\ge R)}
			&\le \|u^+\|_{L^{\chi^{m+1} s}(|x|\ge R-r_{m+1})} \\
			&\le C^{\frac{1}{2}\chi^{-m}} \chi^{m\chi^{-m}} \left(1 + \frac{(R - r_{m+1})^{\frac{3}{t}}}{r_m^2}\right)^{\frac{1}{2\chi^m}} \|u^+\|_{L^{\chi^m s}(|x|\ge R-r_m)} \\
			&\le C^{\frac{1}{2}\chi^{-m}+\frac{1}{2}\chi^{-(m-1)}} \chi^{m\chi^{-m}+(m-1)\chi^{-(m-1)}} (1+2^{2(m+1)})^{\frac{1}{2\chi^m}}\\ &(1+2^{2m})^{\frac{1}{2\chi^{m-1}}} \|u^+\|_{L^{\chi^{m-1} s}(|x|\ge R-r_{m-1})} \\
			&\quad\vdots \\
			&\le C^{\frac{1}{2}\sum_{i=1}^m \chi^{-i}} \chi^{\sum_{i=1}^m i\chi^{-i}} \exp\left(\sum_{i=1}^m \frac{\ln(1+2^{2(i+1)})}{2\chi^i}\right) \|u^+\|_{L^{\chi s}(|x|\ge R-r_1)} \\
			&\le C\|u^+\|_{L^6(|x|\ge \frac{1}{2}R)}.
		\end{align*}
		Letting \( m \to +\infty \) in the last inequality, we obtain
		\[
		\|u^+\|_{L^\infty(|x|\ge R)} \le C\|u^+\|_{L^6(|x|\ge \frac{1}{2}R)}\le \frac{\varepsilon}{2}.
		\]
		By the same argument, we can show that
		\[
		\|u^-\|_{L^\infty(|x|\ge R)} \le C\|u^-\|_{L^6(|x|\ge \frac{1}{2}R)}\le \frac{\varepsilon}{2},
		\]
		where \( u^- = \max(-u,0) \). Hence, for any \( \varepsilon > 0 \) fixed, choose \( R > 1 \) large enough one infers
		\[
		\|u\|_{L^\infty(|x|\ge R)} \le \varepsilon.
		\]
		This implies that
		\[
		|u(x)| \to 0 \quad \text{as} \quad |x| \to +\infty.
		\]
	\end{proof}
	\noindent
	\textbf{Proof of Theorem \ref{th1.4}}  
	By Lemma \ref{lem5.12} and $(f_1)$, see \cite{MR4143621,Moroz2013}, we have
	\[
	\lim_{|x| \to \infty} \frac{f(x,u)u}{|u|^2} = 0.
	\]
	Thus, there exists $\rho > R_1 > 0$ such that for $|x| \geq \rho$
	\[
	\frac{f(x,u)u}{|u|} \leq \frac{3M_1}{4}|u|.
	\]
	Let
	\begin{equation*}
		\operatorname{sign}(u) = 
		\begin{cases}
			\dfrac{u}{|u|}, & \text{if } u(x) \neq 0, \\[1em]
			0, & \text{if } u(x) = 0.
		\end{cases}
	\end{equation*}
	Then Kato's inequality \cite{Kato1972} states that $-\Delta |u| \leq -\operatorname{sign}(u)\Delta u$, which yields that in $\mathbb{R}^3 \setminus B_\rho$
	\begin{align*}
		-\Delta |u| &\leq -\operatorname{sign}(u)\Delta u \\
		&= \frac{u}{|u|} \bigl( -V(x)u -\phi_u u+ f(x,u) \bigr) \\
		&= -V(x)|u| -\phi_u |u|+ \frac{f(x,u)u}{|u|} \\
		&\le -M_1|u| + \frac{3M_1}{4}|u| = -\frac{M_1}{4}|u|.
	\end{align*} 
	That is
	\[
	-\Delta |u| + \frac{M_1}{4}|u| \leq 0 \quad \text{for } |x| \geq \rho.
	\]
	
	By Lemma 6.4 in \cite{Moroz2013}, there exists a nonnegative radial function $v \in C^2(\mathbb{R}^3 \setminus B_\rho)$ such that
	\[
	\begin{cases}
		-\Delta v + \frac{M_1}{4}v = 0, & \text{if } x \in \mathbb{R}^3 \setminus B_\rho, \\[4pt]
		v(x) = |u(x)|, & \text{if } |x| = \rho, \\[4pt]
		v(x) \to 0, & \text{as } |x| \to +\infty.
	\end{cases}
	\]
	Moreover, there exists $C > 0$ such that for $x \in \mathbb{R}^3 \setminus B_\rho$
	\[
	v(x) \leq C|x|^{-1} e^{-\frac{\sqrt{M_1}}{2}|x|}.
	\]
	Hence, by the comparison principle, for sufficiently large $|x|$
	\[
	|u(x)| \leq v(x) \leq C|x|^{-1} e^{-\frac{\sqrt{M_1}}{2}|x|}.
	\]
	This completes the proof of Theorem \ref{th1.4}.
	\qed
	\subsection{A ground state solution}
	In this section, to show the existence of a ground state solution of the system \eqref{eq:4}, we first construct a
	suitable constraint set
	\[S=\{u \in X\setminus\{0\}|J'(u)=0\}\]
	and let
	\[m=\inf_{u \in S} J(u).\]
	From Theorem \ref{th1.1}, there exists \(\bar u\) such that \(J(\bar u)<+\infty\). Hence
	\[m=\inf_{u \in S} J(u)\le J(\bar u)<+\infty.\]
	In the next lemma, we will prove that $m>-\infty$.
	\begin{lem}
		Under the assumptions of Theorem 1.4, we have $m>-\infty$.
	\end{lem}
	\begin{proof}
		By contradiction, we assume that $m=-\infty$. Then there exists $\{u_n\}\subset S$ such that $J(u_n)\to -\infty$ and $J'(u_n)=0$.
		From $(f_2)$ and \eqref{f3.11}, we have
		\begin{align*}
			J(u_n)
			&= \frac{1}{2} \int |\nabla u_n|^2 +\frac{1}{2} \int V(x) u_n^2 + \frac{1}{4} \int \phi_{u_n} u_n^2 - \int F(x,u_n) \\
			&= \frac{1}{2} \int f(x,u_n)u_n - \frac{1}{2} \int \phi_{u_n} u_n^2 + \frac{1}{4} \int \phi_{u_n} u_n^2 - \int F(x,u_n) \\
			&= -\frac{1}{4} \int \phi_{u_n} u_n^2 + \frac{1}{2} \int f(x,u_n)u_n - \int F(x,u_n) \\
			&\ge \frac{1}{4} \int f(x,u_n)u_n - \frac{a_0}{4} \|u_n\|^4 \\
			&\ge -\frac{a_0}{4} \|u_n\|^4,
		\end{align*}
		then $\|u_n\|\to +\infty$.  Let \(v_n = \|u_n\|^{-1} u_n\), then
		\[
		v_n = v_n^+ + v_n^- \rightharpoonup v = v^+ + v^- \in X, \quad v_n^\pm, v^\pm \in X^\pm.
		\]
		If \(v = 0\), then \(v_n^- \to v^- = 0\) because \(\dim X^- < \infty\). Since
		\[
		\|v_n^+\|^2 + \|v_n^-\|^2 = 1,
		\]
		for \(n\) large enough we have
		\[
		\|v_n^+\|^2 - \|v_n^-\|^2 \geq \frac{1}{2}.
		\]
		Hence
		\begin{align*}
			J(u_n) - \frac{1}{4} \langle J'(u_n), u_n \rangle &= \frac{1}{4} \|u_n\|^2 \left( \|v_n^+\|^2 - \|v_n^-\|^2 \right) + \int \left( \frac{1}{4} f(x, u_n) u_n - F(x, u_n) \right) \\
			&\geq  \frac{1}{8}\|u_n\|^2,
		\end{align*}
		contradicting $\|u_n\|\rightarrow+\infty$.
		
		If $v\neq 0$, the set $A=\{x\in \mathbb{R}^3: v(x)\neq 0\}$ has positive Lebesgue measure. For $x\in A$ we have $|u_n(x)|\rightarrow+\infty$ and
		\begin{align*}
			\frac{F(x, u_n(x))}{u_n^4(x)} v_n^4(x) \to +\infty,
		\end{align*}
		thanks to \eqref{eq:5}. By Fatou lemma, for $x\in A$ we get
		\[\int\frac{F(x, u_n)}{u_n^4} v_n^4 \ge \int_A \frac{F(x, u_n)}{u_n^4} v_n^4 \to +\infty.\]
		It follows from $\langle J'(u_n),u_n\rangle\rightarrow 0$ and  Lemma \ref{lem3.1} that
		\begin{align*}
			\int_A \frac{4F(x,u_n)}{u_n^4} v_n^4&\leq\int_A \frac{f(x,u_n)u_n}{u_n^4} v_n^4= \frac{1}{\|u_n\|^4} \int_A f(x,u_n)u_n \\
			&\leq \frac{1}{\|u_n\|^4} \int f(x,u_n)u_n \\
			&= \frac{1}{\|u_n\|^4} \left(\int \left(|\nabla u_n|^2 + V(x)u_n^2\right) + \int \phi_{u_n} u_n^2 - \langle J'(u_n),u_n\rangle\right) \\
			&\leq 2+ \frac{1}{\|u_n\|^4} \int \phi_{u_n} u_n^2 \\
			&\leq 2 + a_0,
		\end{align*}
		a contradiction. Thus $\{u_n\}$ is bounded in $X$ and $m>-\infty$.
	\end{proof}
	To show that the chosen minimizing sequence does not converge to $0$, the following lemma plays a key role in the proof.
	\begin{lem}\label{8.1}
		The zero function $0$ is an isolated critical point of $J$.
	\end{lem}
	\begin{proof}
		By a direct computation, for any $u \in S$, where $u=u^++u^-$ with $u^\pm \in X^\pm$. We have
		\begin{align*}
			\langle J'(u), u^+ \rangle = \|u^+\|^2+\int \phi_{u}uu^+-\int f(x,u)u^+=0.
		\end{align*}
		Similarly,
		\[\langle J'(u), u^- \rangle=-\|u^-\|^2+\int\phi_uuu^-- \int f(x,u)u^-=0.\]
		From \eqref{f3.01}, Sobolev embedding inequality and  H\"older inequality, it is clear to see that
		\begin{align}\label{6.21}
			\| u^+ \|^2
			&= -\int \phi_uuu^+ + \int f(x,u) u^+ \nonumber \\
			&\le \|\phi_u\|_6\|u\|_2\|u^+\|_3 + \varepsilon \|u\|_2 \|u^+\|_2 + C_\varepsilon \|u\|_6^5 \|u^+\|_6 \nonumber \\
			&\le C\left( \|u\|^3 \|u^+\| + \varepsilon \|u\|\|u^+\|+ C_\varepsilon \|u\|^5 \|u^+\| \right).
		\end{align}
		Since $\|u\|^2=\|u^+\|^2+\|u^-\|^2$, we have $\|u\|\ge\|u^+\|$, $\|u\|\ge\|u^-\|$. It follows that
		\begin{align}
			\| u^+ \|^2\le C\left( \|u\|^4 + \varepsilon \|u\|^2+ C_\varepsilon \|u\|^6\right)
		\end{align}
		Similarly,
		\begin{align}\label{6.22}
			\| u^- \|^2
			\le  C\left( \|u\|^4 + \varepsilon \|u\|^2+ C_\varepsilon \|u\|^6\right).
		\end{align}
		Taking $\varepsilon=\frac{1}{4C}$, it follows from \eqref{6.21} and \eqref{6.22} that
		\begin{align*}
			\frac{1}{2}\|u\|^2 
			\le C\left( \|u\|^4 + C_\varepsilon \|u\|^6 \right).
		\end{align*}
		This implies that $\|u\|\ge \delta$ for some $\delta > 0$. 
	\end{proof}
	\noindent
	\textbf{Proof of Theorem \ref{th1.5}} 
	Taking a minimizing sequence $\{u_n\} \subset S$ such that $J'(u_n) = 0$ and $J(u_n) \to m$. Then  $\{u_n\}$ is the $(PS)$ sequence, i.e.,
	\[\sup_n |J(u_n)| < +\infty \quad \text{and} \quad J'(u_n) = 0.\]
	Then, by Lemmas \ref{lem3.2} and \ref{lem3.5}, we know that this sequence is bounded and converges strongly in $X$.  
	Thus, there exists $u_0$ such that $u_n \to u_0$ in $X$. By Lemma \ref{8.1}, we also get $\|u_0\|\ge \delta$ for some $\delta > 0$. Consequently, we have $u_0\neq 0$ and $u_0 \in S$. Therefore, $u_0$ is a ground state solution.
	\qed
	\section{Proof of Theorem \ref{th1.3}}
	In this section, we assume that $f(x,u)$ is odd with respect to $u$ and prove the existence of multiple solutions via the symmetric mountain pass theorem of Ambrosetti–Rabi-nowitz \cite{MR0345031}.
	
	\begin{prop}\label{thm6.1}\cite{Rabinowitz1984}
		Let $X$ be an infinite dimensional Banach space, $J \in C^1(X,\mathbb{R})$ be even, satisfies (PS) condition and $J(0)=0$. If $X=Y \oplus Z$ with $dim Y<\infty$. Moreover, $J$ satisfies
		\item[(i)] there are constants $\rho,\alpha > 0$, such that $J|_{\partial B_\rho \cap Z} \ge \alpha$;
		\item[(ii)] for any finite dimensional subspace $W\subset X$, there is an $R=R(W)$ such that $J\le0$ on $W\setminus B_{R(W)}$.
		\par\noindent Then $J$ has a  sequence of critical values $c_j \to \infty$.
	\end{prop}
		From $(V_1)$, we set $Y:=X^-$ and $ Z:=X^+$, then $X=Y \oplus Z$. Moreover, there exists a positive constant $\lambda\in (0,1)$ such that
	\begin{equation*} 
		Q(u) \ge \lambda\|u\|^2, \quad \text{for all } u \in Z.
	\end{equation*}  
	\begin{lem}\label{lem6.2}
		Suppose that $(V_1)$, $(f_1)$ and $(f_2)$ hold,
		\item[(i)] there are constants $\rho,\alpha > 0$, such that $J|_{\partial B_\rho \cap Z} \ge \alpha$;
		\item[(ii)]let $W$ be a finite dimensional subspace of $X$, then $J$ is anti-coercive on $W$.
	\end{lem}
	
	\begin{proof}
		\item[(i)]
		Let $u \in Z$ and $\|u\|=\rho$.
		Then by \eqref{f3.01}
		\begin{align*}
			J(u) &= \frac{1}{2}Q(u) + \frac{1}{4}\int_{\mathbb{R}^3} \phi_u u^2 \,\mathrm{d}x - \int_{\mathbb{R}^3} F(x,u) \,\mathrm{d}x \\
			&\geq  \frac{\lambda}{2} \|u\|^2 - \frac{1}{2}\varepsilon \|u\|_2^2 - \frac{1}{6}C_\varepsilon \|u\|_6^6 \\
			&\geq \left( \frac{\lambda}{2}-\frac{1}{2}\varepsilon C_2^2-\frac{1}{6}C_6^6 C_\varepsilon\|u\|^4\right)\|u\|^2.
		\end{align*}
		We choose $\varepsilon=\frac{\lambda}{2C_2^2} > 0$ such that $\frac{\lambda}{2}-\frac{1}{2}\varepsilon C_2^2=\frac{\lambda}{4}>0$ and let $\frac{1}{6}C_6^6 C_\varepsilon=C$. Then for sufficiently small $\rho$, there exists $\alpha$ such that
		\[
		J(u)\ge\left( \frac{\lambda}{4}-C\rho^4\right)\rho^2\ge \alpha,
		\]for any $u \in \partial B_\rho \cap Z $.
		\item[(ii)] Suppose that $\{u_n\} \subset W$ and $\|u_n\| \to \infty$, let $h_n=\|u_n\|^{-1}u_n$, then $ \|h_n\|=1$. Thus, $\{h_n\}$ is bounded. Because of $\text{dim} W<\infty$, there exists $h \in W\setminus \{0\}$, such that
		\[
		h_n \to h \in  W\setminus \{0\} \quad \text{and} \quad h_n(x) \to h(x)\quad \text{a.e. in} \quad \mathbb{R}^3.
		\]
		Similar to \eqref{f3.2}, we obtain
		\begin{equation}\label{6.2}
			\int\frac{F(x, u_n(x))}{\|u_n\|^4} \to +\infty.
		\end{equation}
		Combining \eqref{6.2} and \eqref{f3.11}, we have
		\begin{align*}
			J(u_n)&= \frac{1}{2}Q(u_n) + \frac{1}{4}\int \phi_{u_n} u_n^2 - \int_{\mathbb{R}^3} F(x,u_n) \\
			&\le \|u_n\|^4\left\{\frac{1}{2\|u_n\|^2}+\frac{a_0}{4}-\int\frac{F(x, u_n(x))}{\|u_n\|^4}\right\} \to -\infty.
		\end{align*}
		The proof of the lemma is complete.
	\end{proof}
	\noindent
	\textbf{Proof of Theorem \ref{th1.3}} By Lemma \ref{lem6.2}, parts (i) and (ii) of Proposition \ref{thm6.1} hold. Moreover, since $f(x,u)$ is odd, the functional $J \in C^1(X,\mathbb{R})$ is even. And Lemma \ref{lem3.5} guarantees that $J$ satisfies $(PS)$ condition. Consequently, by Proposition \ref{thm6.1}, there exists a sequence of solutions $(u_n, \phi_{u_n}) \in X \times \mathcal{D}$ such that each $u_n$ is a critical point of $J$ and
	\[
	J(u_n) \to +\infty \quad \text{as } n \to \infty.
	\]
	Moreover, $\phi_{u_n}$ is the unique solution of the second equation in \eqref{eq:4} corresponding to $u_n$.\qed
\section*{Acknowledgments} 
L. F. Yin was supported by the Natural Science Foundation of Sichuan Province, China (No. 2024NSFSC1343) and
National Natural Science Foundation of China (No.124\\01140). F. Wang was supported by the Fundamental Research Funds for the Central Universities(2682026ZTPY010), China.

\section*{Author contributions}
All authors contributed equally in writing, analyzing, and reviewing this manuscript.

\section*{Data availability}
No data was used for the research described in the article.

\section*{Conflict of interest}
The authors state no conflict of interest.

	\providecommand{\href}[2]{#2}
	\providecommand{\arxiv}[1]{\href{http://arxiv.org/abs/#1}{arXiv:#1}}
	\providecommand{\url}[1]{\texttt{#1}}
	\providecommand{\urlprefix}{URL }


\begin{thebibliography}{1}
		
		\bibitem{MR0345031}
		Ambrosetti,~A., Rabinowitz,~P. H. :
		\newblock Dual variational methods in critical point theory and applications,
		\newblock \emph{J. Funct. Anal.}, \textbf{14},  349-381 (1973).
		
		\bibitem{MR3998534}
		\newblock Bonin,~C. A., Pimentel,~B. M., Ortega,~P.~H. :
		\newblock Multipole expansion in generalized electrodynamics,
		\newblock \emph{Internat. J. Modern Phys. A}, \textbf{34}, 1950134 (2019).
		
		\bibitem{Brezis1978}
		\newblock Brezis,~H., Kato,~T. :
		\newblock Remarks on the Schroedinger operator with singular complex potentials. Technical summary report,
		\newblock \textbf{}(1978).
		
		\bibitem{MR3689587}
		\newblock Bertin,~M. C., Pimentel,~B. M., Valc\'arcel,~C. E., Zambrano,~G.~E.~R. :
		\newblock Hamilton-{J}acobi formalism for {P}odolsky's electromagnetic theory on the null-plane, 
		\newblock \emph{J. Math. Phys.}, \textbf{58}, 082902 (2017).
		
		\bibitem{PhysRevD.90.085012}
		\newblock Bufalo,~R., Pimentel,~B. M., Soto,~ D. E. :
		\newblock Causal approach for the electron-positron scattering in generalized quantum electrodynamics,
		\newblock \emph{Phys. Rev. D}, \textbf{90}, 085012 (2014).
		
		\bibitem{MR3707699}
		\newblock Bufalo,~R., Pimentel,~B. M., Soto,~ D. E. :
		\newblock Normalizability analysis of the generalized quantum electrodynamics from the causal point of view,
		\newblock \emph{Int. J. Mod. Phys. A}, \textbf{32}, 1750165 (2017).
		
		\bibitem{MR1420790}
		\newblock Bartsch,~T., Li,~S. J. :
		\newblock Critical point theory for asymptotically quadratic functionals and
		applications to problems with resonance,
		\newblock \emph{Nonlinear Anal.}, \textbf{28}, 419--441 (1997).
		
		\bibitem{Bartsch2004}
		\newblock Bartsch,~T., Liu,~Z. L., Weth,~T. :
		\newblock Sign changing solutions of superlinear {S}chr\"odinger equations,
		\newblock \emph{Comm. Partial Differential Equations}, \textbf{29}, 25--42 (2004).
		
		\bibitem{MR4119256}
		Chen,~H. Y., Liu,~S. B. :
		\newblock Standing waves with large frequency for 4-superlinear {S}chr\"odinger--{P}oisson systems,
		\newblock \emph{Ann. Mat. Pura Appl. (4)}, \textbf{194}, 43--53 (2015).
		\newblock 
		
		\bibitem{MR1196690}
		\newblock Chang,~K. C. :
		\newblock Infinite-dimensional {M}orse theory and multiple solution problems,
		\newblock\emph{Progress in Nonlinear Differential Equations and their
			Applications}, \textbf{6}
		\newblock Birkh\"{a}user Boston, Inc., Boston, MA, (1993).
		
		\bibitem{ChenLiRadulescuTang2021}
		\newblock Chen,~S. T., Li,~L., R\u{a}dulescu,~V. D., Tang,~X. H. :
		\newblock Ground state solutions of the non-autonomous Schr\"odinger--Bopp--Podolsky system,
		\newblock \emph{Anal. Math. Phys.} \textbf{12}, 27 (2022).
		
		\bibitem{MR4050783}
		\newblock Chen,~S. T., Tang,~X. H. :
		\newblock On the critical Schr\"odinger--Bopp--Podolsky system with general nonlinearities,
		\newblock \emph{Nonlinear Anal.} \textbf{195} , 111734 (2020).
		
		\bibitem{MR5014952}
		\newblock Du,~Y. Q., Su,~Y. :
		\newblock Ground state solution of {S}chr\"odinger--{B}opp--{P}odolsky system in the mass subcritical case, 
		\newblock \emph{Appl. Math. Lett.},\textbf{176}, 109870 (2026).
		
		\bibitem{MR3957980}
		\newblock D'Avenia,~P., Siciliano,~G. :
		\newblock Nonlinear {S}chr\"odinger equation in the {B}opp-{P}odolsky
		electrodynamics: solutions in the electrostatic case,
		\newblock \emph{J. Differential Equations}, \textbf{267}, 1025--1065 (2019). 
		
		\bibitem{MR4626835}
		\newblock Figueiredo,~G. M., Siciliano,~G. : 
		\newblock Multiple solutions for a {S}chr{\"o}dinger-{B}opp-{P}odolsky system with positive potentials, 
		\newblock \emph{Math. Nachr.}, \textbf{296}, 2332--2351 (2023).
		
		\bibitem{Gilbarg2001}
		\newblock Gilbarg,~D., Trudinger,~N. S. :
		\newblock Elliptic partial differential equations of second order, reprint of the 1998th edn.
		\newblock \emph{Classics in Mathematics}, Springer-Verlag, Berlin (2001).
		
		\bibitem{MR4889191}
		\newblock Huang,~J., Wang,~S. :
		\newblock Normalized ground states for the mass supercritical Schr\"odinger--Bopp--Podolsky system: existence, uniqueness, limit behavior, strong instability,
		\newblock \emph{J. Differential Equations}, \textbf{437}, 113282 (2025).
		
		\bibitem{Jiang2023}
		\newblock Jiang,~S., Liu,~S. B. :
		\newblock Standing waves for 6-superlinear {C}hern-{S}imons-{S}chr\"odinger
		systems with indefinite potentials,
		\newblock \emph{Nonlinear Anal.}, \textbf{230}, 113234 (2023).
		
		\bibitem{Kato1972}
		\newblock Kato,~T. :
		\newblock Schr{\"o}dinger operators with singular potentials,
		\newblock \emph{Israel J. Math.}, \textbf{13}, 135–148 (1972).
		
		\bibitem{Liu2022}
		\newblock Liu,~C. :
		\newblock Existence and stability of standing waves with prescribed $L^2$-norm for a class of Schrödinger-Bopp-Podolsky system, 
		\newblock \emph{J. Appl. Math. Phys.}, \textbf{10}, 2245–2267 (2022).
		
		\bibitem{LiebLoss2001}
		\newblock Lieb,~E., Loss,~M. :
		\newblock Analysis,
		\newblock \emph{ Grad. Stud. Math.} \textbf{14}, American Mathematical Society, Providence, RI (2001).
		
		\bibitem{1990Some}
		\newblock Li,~G. B. :
		\newblock Some properties of weak solutions of nonlinear scalar field equations,
		\newblock \emph{Ann. Fenn. Math.}, \textbf{15}, 25--36 (1990).
		
		\bibitem{MR1110119}
		\newblock Liu,~J. Q. :
		\newblock The {M}orse index of a saddle point,
		\newblock \emph{Systems Sci. Math. Sci.}, \textbf{2}, 32--39 (1989).
		
		\bibitem{MR4081540}
		\newblock Liu,~S. B., Mosconi,~S.:
		\newblock On the Schr\"odinger--Poisson system with indefinite potential and 3-sublinear nonlinearity,
		\newblock \emph{J. Differential Equations} \textbf{269}, 689--712 (2020).
		
		\bibitem{MR3656292}
		\newblock Liu,~S. B., Wu,~Y. :
		\newblock Standing waves for 4-superlinear {S}chr\"odinger-{P}oisson systems with indefinite potentials, 
		\newblock \emph{Bull. Lond. Math. Soc.}, \textbf{49}, 226--234 (2017).
		
		\bibitem{LI19956} 
		\newblock Li,~S. J., Willem,~M. : 
		\newblock Applications of {L}ocal {L}inking to {C}ritical {P}oint {T}heory, 
		\newblock \emph{J. Math. Anal. Appl.}, \textbf{189}, 6--32(1995).
		
		\bibitem{MR982267}
		\newblock Mawhin,~J., Willem,~M. :
		\newblock Critical point theory and {H}amiltonian systems,
		\newblock \emph {Appl. Math. Sci.}, 
		\textbf{74}
		\newblock Springer-Verlag, New York, 1989.
		
		\bibitem{Moroz2013}
		\newblock Moroz,~V., Van Schaftingen,~J. :
		\newblock Groundstates of nonlinear Choquard equations: Existence, qualitative properties and decay asymptotics,
		\newblock \emph{Journal of Functional Analysis}, \textbf{265}, 153--184 (2013).
		
		\bibitem{Rabinowitz1984}
		\newblock Rabinowitz,~P. H. : 
		\newblock Minimax methods and their application to partial differential equations,
		\newblock in \emph{Seminar on Nonlinear Partial Differential Equations (Berkeley, Calif., 1983)},
		\newblock Math. Sci. Res. Inst. Publ., Springer, New York, 307--320 (1984).
		
		\bibitem{MR4119258}
		\newblock Siciliano,~G., Silva,~K. :
		\newblock The fibering method approach for a non-linear {S}chr\"odinger equation coupled with the electromagnetic field,
		\newblock \emph{Publ. Mat.}, \textbf{64}, 373--390 (2020).
		
		\bibitem{MR4982514}
		\newblock Siciliano,~G. :
		\newblock Normalized solutions for {S}chr{\"o}dinger-{B}opp-{P}odolsky systems in bounded domains,
		\newblock \emph{Commun. Math.}, \textbf{34}, 3 (2026).
		
		\bibitem{MR4198483}
		\newblock Teng,~K. M., Yan,~Y. X. :
		\newblock Existence of a positive bound state solution for the nonlinear {S}chr\"odinger-{B}opp-{P}odolsky system,
		\newblock \emph{Electron. J. Qual. Theory Differ. Equ.}, \textbf{}19 (2021).
		
		\bibitem{MR4143621}
		\newblock Xia,~J.~K., Wang,~Z. Q. :
		\newblock Saddle solutions for the Choquard equation,
		\newblock \emph{Calc. Var. Partial Differential Equations}, \textbf{58}, 85 (2019).
		
		\bibitem{MR4143620}
		\newblock Yang,~J., Chen,~H. B., Liu,~S. L. :
		\newblock The existence of nontrivial solution of a class of {S}chr\"odinger-{B}opp-{P}odolsky system with critical growth,
		\newblock \emph{Bound. Value Probl.}, \textbf{}144 (2020).
	
		\bibitem{MR4519835} 
		\newblock Yang,~H., Yuan,~Y. X., Liu,~J. : 
		\newblock On nonlinear {S}chr{\"o}dinger-{B}opp-{P}odolsky system with asymptotically periodic potentials, 
		\newblock \emph{J. Funct. Spaces}, \textbf{}, 9287998 (2022). 
		
		\bibitem{MR4944207}
		\newblock Zheng,~H. N., Li,~L., Chen,~S. J., O'Regan,~D. :
		\newblock Ground state sign-changing solutions for asymptotically cubic or super-cubic Schr\"odinger--Bopp--Podolsky systems,
		\newblock \emph{Rend. Circ. Mat. Palermo (2)} \textbf{74}, 28 (2025).
		
		\bibitem{MR4217993} 
		\newblock Zhu,~ Y. T., Chen,~C. F., Chen,~J.~H. : 
		\newblock The {S}chr{\"o}dinger-{B}opp-{P}odolsky equation under the effect of nonlinearities, 
		\newblock \emph{Bull. Malays. Math. Sci. Soc.}, \textbf{44}, 953--980 (2021). 
		
		
	\end{thebibliography}
\end{document}